\documentclass{amsart}
\usepackage{amsthm,amsfonts,amsthm,latexsym,amsmath,amssymb,amscd,amsmath, mathrsfs, epsf,amsmath,mathrsfs}

\usepackage{graphicx}
\theoremstyle{plain}
\newtheorem{theorem}{Theorem}
\newtheorem*{theorem*}{Theorem}

\newtheorem{lemma}[theorem]{Lemma}

\newtheorem{remark}{Remark}

\newtheorem*{qw*}{Question}

\newtheorem{claim}{Claim}
\newcommand{\be}{\begin{equation}}
\newcommand{\ee}{\end{equation}}

\newcommand {\la}{\lambda}
\newcommand {\te}{\theta}

\newcommand{\tu}{Tur\'{a}n }

\begin{document}

\title{On Tur\'{a}n inequality for ultraspherical polynomials}

\author[I. Krasikov]{Ilia Krasikov}

\address{   Department of Mathematical Sciences,
            Brunel University,
            Uxbridge UB8 3PH United Kingdom}
\email{mastiik@brunel.ac.uk}

\begin{abstract}
        We show that the normalised ultraspherical polynomials, $G_n^{(\lambda)}(x)=C_n^{(\lambda)}(x)/C_n^{(\lambda)}(1)$, satisfy the following stronger version of Tur\'{a}n inequality,
$$|x|^\te \left(G_n^{(\la)}(x)\right)^2 -G_{n-1}^{(\la)}(x)G_{n+1}^{(\la)}(x) \ge 0 ,\;\;\;|x| \le 1, $$
where $\te=4/(2-\la)$ if $-1/2 <\la \le 0$, and $\te=2/(1+2\la)$ if $\la \ge 0$.  We also provide a similar generalisation of Tur\'{a}n  inequalities for some symmetric orthogonal polynomials with a finite or infinite support defined by a three term recurrence.
\end{abstract}

\maketitle

\noindent

\vspace{3mm}
\noindent \emph{2010 Mathematics Subject Classification 33C45}

\section{Introduction}

Let
$$y_n(x)=G_n^{(\la)}(x) =\dfrac{C_n^{(\la)}(x)}{C_n^{(\la)}(1)} $$
be the normalised ultraspherical polynomials, $G_n^{(\la)}(1)=1$.
The classical \tu inequality states \cite{sz1} that
\be
\label{tur0}
y_n^2(x) -y_{n-1}(x) y_{n+1}(x) \ge 0,  \;\; \la \ge 0, \;\; |x| \le 1.
\ee
Gerhold and Kauers \cite{GK}  gave a computer proof of the following stronger result in the special case $\la=\frac{1}{2}$ of Legendre polynomials $P_n(x)$,
\be
\label{leg0}
|x|P_n^2(x)-P_{n-1}(x)P_{n+1}(x)  \ge 0 , \;\;\; |x| \le 1.
\ee
In turn, this was extended by Nikolov and Pillwein to the ultraspherical case \cite{GP},
\be
\label{leg1}
|x| y_n^2(x) -y_{n-1}(x) y_{n+1}(x) \ge 0, \;\;\; -1/2 < \la \le 1/2, \;\;|x| \le 1.
\ee

In this paper we use an elementary geometric approach which, in particular, leads  to the following sharper generalization of (\ref{leg1}):
\begin{theorem}
\label{thmain}
For $\la > -1/2$ and $|x| \le 1$ the following \tu type inequality holds,
\be
\label{tumain}
\Delta_n=\Delta_n^{(\la)}(x)=|x|^\te y_n^2(x) -y_{n-1}(x)y_{n+1}(x) \ge 0,
\ee
where
$$\te= \left\{
\begin{array}{cc}
\dfrac{4}{2-\la} \, ,& -1/2 <\la \le 0;\\
&\\
\dfrac{2}{1+2 \la} \,,& \la \ge 0.
\end{array}
\right.
$$
Moreover, for $\la \ge 0$ the above value of
$ \te $ is the best possible for all $n \ge 1$.
\end{theorem}

For $-1/2 <\la < 0$ the obtained result $ \te=\dfrac{4}{2-\la}$  is not sharp and, in fact, the maximal possible value of $\te$ depends on $n$.
Note also that generally  one must have $\te \le 2$ for all results of this type.
This is so because for any family of orthogonal polynomials defined by the general three term recurrence
$$p_{n+1}(x)=(b_n x+c_n)p_n(x)-a_n p_{n-1}(x) ,$$
the following easy to check identity holds for all $x \in \mathbb{R}$:
\begin{equation}
\label{ident}
\frac{(x b_n+c_n)^2}{4 a_n} \, p_n^2(x)-p_{n-1}(x)p_{n+1}(x) =
\frac{(p_{n+1}(x)-a_n p_{n-1}(x))^2}{4a_n} \ge 0.
\end{equation}
Thus, all local maxima of the ratio
$\dfrac{p_{n-1}(x)p_{n+1}(x)}{p_n^2(x)}$ lie on the curve $\dfrac{(x b_n+c_n)^2}{4 a_n}\,$. In particular, (\ref{ident}) is sharp in the orthogonality interval.

For the normalised ultraspherical polynomials, which are  defined by the recurrence
\be
\label{3termr}
(n+2 \la) y_{n+1}=2(n+\la)x y_n-n y_{n-1},
\ee
identity (\ref{ident}) yields
\be
\label{univ}
\left(1+\dfrac{\la^2}{n(n+2\la)} \right)x^2 y_n^2-y_{n-1}y_{n+1} \ge 0 , \;\;\; x \in \mathbb{R}.
\ee
Thus, the inequality
$$x^2 y_n^2-y_{n-1}y_{n+1} \ge 0, \;\; x \in [-1,1],$$
 is possible only for $\la=0$, that is, for Chebyshev $T_n$ polynomials.

Our proof of Theorem \ref{thmain} can be extended to some other families of orthogonal polynomials. In particular, we provide the following generalisation
of the case $\lambda < 0$ of Theorem \ref{thmain}.
Consider a family of symmetric polynomials $p_n(x)$ orthogonal on $[-1,1]$ and  normalised by $p_n(1)=1$.
Let $p_{-1}=0,\; p_0=1$ and $a_0=0$, then such
polynomials are defined by the following the three term recurrence,
\begin{equation}
\label{gen3term}
(1-a_n)p_{n+1}(x)=x p_n(x)- a_n p_{n-1}(x), \;\;\; 0 <a_n <1.
\end{equation}
\begin{theorem}
\label{genleg}
Suppose that the sequence $a_n$ in (\ref{gen3term}) is decreasing and $\frac{1}{2} <a_n <1$.
Then the following \tu type inequality holds,
$$\Delta_n = |x|^\te p_n^2(x)-p_{n-1}(x)p_{n+1}(x) \ge 0, \;\;\;|x| \le 1, $$
where
\be
\label{gen te}
\te= \inf\limits_{n \ge 1} \dfrac{2 \log \dfrac{(1-a_n)\,a_{n+1}}{(1-a_{n+1})\,a_n}}{\log \dfrac{4(1-a_n)\, a_{n+1}^2}{a_n}} .
\ee
\end{theorem}

In the ultraspherical case and $-\frac{1}{2} <\la \le 0$, this gives the same value for $\te$ as Theorem \ref{thmain}.

We will also consider orthogonal polynomials with (possibly) unbounded support, the family which includes Hermite polynomials as a special case.
It was observed in \cite{askey}  that the \tu inequality
\begin{equation}
\label{turask}
p_n^2(x)-p_{n-1}(x)p_{n+1}(x) \ge 0
\end{equation}
holds for any family of monic symmetric orthogonal polynomials $p_n(x)$ defined by the three term recurrence
\be
\label{sym3r}
p_{n+1}(x)=x p_n(x)-a_n p_{n-1}(x), \; \; p_{-1}=0, \; p_0=1, \;\;\; a_n > 0,
\ee
with the increasing coefficients $a_n$.
We will show that (\ref{turask}) can also be strengthened in the spirit of Theorem \ref{thmain}. Note that in the case of an unbounded support, that is, $a_n \rightarrow \infty$ in (\ref{sym3r}),
any inequality of the form
$$\phi(x) p_n^2(x)-p_{n-1}(x)p_{n+1}(x) \ge 0,$$
must satisfy
$\lim\limits_{x \rightarrow \infty}\phi(x) =1$, as
$$\lim\limits_{x \rightarrow \infty} \frac{p_{n-1}(x)p_{n+1}(x)}{p_n^2(x)} =1,$$
and possibly even stronger restrictions due to the asymptotic behaviour of $p_n(x)$.
\begin{theorem}
\label{thherm}
Let $p_n(x)$ be a family of monic symmetric orthogonal polynomials satisfying (\ref{sym3r}) with an increasing sequence $a_n, \; n=1,2,...\,$.
Then the following \tu type inequality holds,
\be
\label{genhermite}
\Delta_n(x)=\frac{x^2}{x^2+a_{n}-a_{n-1}}\, p_n^2(x)-p_{n-1}(x)p_{n+1}(x) \ge 0, \;\;\; x \in \mathbb{R}.
\ee
In particular, for the Hermite polynomials in the monic or standard normalisation \cite{szego} we have
\be
\label{hermite}
\frac{x^2}{x^2+\frac{1}{2}}H_n^2(x)-H_{n-1}(x) H_{n+1}(x) \ge 0.
\ee
\end{theorem}
It would be interesting to obtain similar inequalities for Laguerre type polynomials orthogonal on $[0,\infty)$.

Our approach is based on the following observation. Setting
\begin{equation}
\label{defint}
t=t_n(x) = \frac{p_{n+1}(x)}{p_n(x)}
\end{equation}
and using the corresponding three term recurrence, one can rewrite
a \tu type inequality as a quadratic form $Q(x)=A(x)t^2+B(x)t+C(x)$ in $t$, where we have to assume that $A(x) > 0$, since otherwise the form attains negative values in the oscillatory region, where $t$ runs over all real values.  Suppose now that one wants to prove that $Q_{n+1}(x)=A_{n+1}(x)t^2+B_{n+1}(x)t+C_{n+1}(x) \ge 0$, using that  $Q_n(x)=A_n(x)t^2+B_n(x)t+C_n(x) \ge 0$.
Consider the corresponding curves defined by the equalities
$$\mathcal{T}_n(x)=A_n(x)\tau^2+B_n(x)\tau+C_n(x)=0,$$
and
$$\mathcal{T}_{n+1}(x)=A_{n+1}(x)\tau^2+B_{n+1}(x)\tau+C_{n+1}(x)=0. $$
First, one checks that $\mathcal{T}_n(x)$ and $\mathcal{T}_{n+1}(x)$ do not intersect by, say, showing that the resultant of the curves in $\tau$  does not vanish.
Next, since $ Q_n(x)\ge 0$, quadratic $Q_{n+1}(x)$ could be negative only inside the region enclosed by $\mathcal{T}_n(x)$. Therefore, it is enough to check that the curves are nested with $\mathcal{T}_{n+1}(x)$ lying inside  $\mathcal{T}_n(x)$. Fig.1 illustrates the above arguments for the ultraspherical case.

\begin{figure}[t]

\includegraphics[scale=0.8]{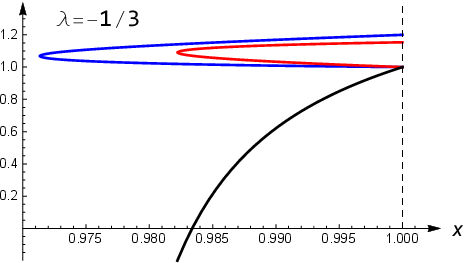}
\includegraphics[scale=0.8]{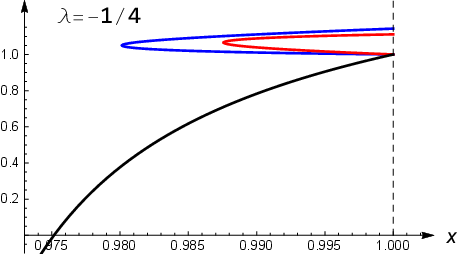}

\vspace{3mm}
\includegraphics[scale=0.8]{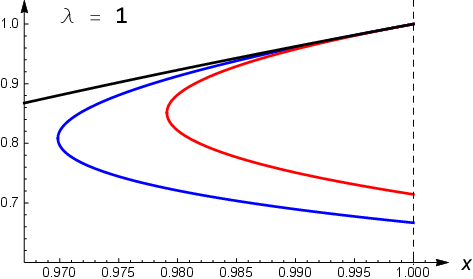}

\caption{The black line is $t_4$; the blue and red lines show $\mathcal{T}_4$ and $\mathcal{T}_5$, respectively, for three different values of $\la$. }

\end{figure}

\section{Proof of Theorem \ref{thmain}}

We will proceed by the induction on $n \ge 1$. Without loss of generality we assume that $\te <2$ and $x \in [0,1]$. Furthermore, when it is convenient, we may restrict the value of $x$ to open interval $(0,1)$ treating the endpoints $x=0,1$ as the limiting case.

 Assume first that $-\frac{1}{2} <\la < 0$ and check that $\Delta_1 >0$. In this case already the assumption $\te <2$ will suffice.
We have
$$(1+2 \la)\Delta_1=1-2(1+\la)x^2+(1+2\la)x^{2+\te}
 \ge 1-2(1+\la)x^2+(1+2\la)x^4 =$$
 $$(1-x^2)\left(1-(1+2\la)x^2 \right)>0.$$
For $ \la >0$ and $\te=\frac{2}{1+2\la}$, noticing that
$$ (1+2 \lambda)\frac{d \Delta_1}{d x}= -4(1+\lambda)(1-x^\te) x \le 0,$$
one finds
$$\Delta_1(x)> \Delta_1(1)=0.$$
Note that for $n=1$ the chosen value $\te= \frac{2}{1+2\la}$ is the best possible because the expansion of $\Delta_1$ into the Taylor series around $x=1$ is
$$(1+2 \lambda)\Delta_1(x)=\left( 2-(1+2\lambda) \te \right)(1-x) +O((1-x)^2). $$

More work is needed for the induction step in order to show that
$\Delta_n >0$ implies $\Delta_{n+1} >0$. We split the proof into a few lemmas.

As in (\ref{defint}), for fixed $n$ and $\lambda$ using we introduce the function
$$t=t_n^{(\la)}(x)= \dfrac{G_{n+1}^{(\la)}(x)}{G_{n}^{(\la)}(x)} \, .$$
Using (\ref{3termr}), we can rewrite $\Delta_n$ and $\Delta_{n+1}$ as follows,

\be
\label{w0}
\frac{n \Delta_n}{\left(G_n^{(\la)}(x) \right)^2}=( n+2\la) t^2 - 2 ( n+\la) x t +  n x^\te,
\ee

\be
\label{w1}
 \frac{(n+2\la+1)\Delta_{n+1}}{\left(G_n^{(\la)}(x) \right)^2}=(n+2\la+1)x^\te t^2-2(n+\la+1)x t+n+1.
\ee
Now consider the corresponding curves
$$ \mathcal{T}_n=( n+2\la) \tau_n^2 - 2 ( n+\la) x \tau_n +  n x^\te =0,$$
and
$$ \mathcal{T}_{n+1}=(n+2\la+1)x^\te \tau_{n+1}^2-2(n+\la+1)x  \tau_{n+1}+n+1 =0,$$
or explicitly,
\begin{equation}
\label{tn}
\tau_n= \frac{(n+\la) x \pm \sqrt{(n+\la)^2 x^2 -  n (n+2\la)x^\te}}{n+2\la} \,,
\end{equation}
\begin{equation}
\label{tn1}
\tau_{n+1}= \frac{(n+\la+1) x \pm \sqrt{(n+\la+1)^2 x^2 -( n+1) (n+2 \la+1)x^\te}}{(n+2\la+1)x^\te}\, ,
\end{equation}
where
$0 \le x \le 1$. Both curves have a single component and are simultaneously real for
\be
\label{x0}
x \ge x_0=\left( \frac{(n+1)(n+2\la+1)}{(n+\la+1)^2} \right)^{1/(2-\te)} \, ,
\ee
where $x_0$ is just the vertex of $\mathcal{T}_{n+1}$.  Thus, for $x \le x_0$ both functions $\Delta_n$ and $\Delta_{n+1}$ are nonnegative, the first one by the induction hypothesis. So (as it will be necessary for $\lambda <0$), we may restrict the proof to $x >x_0$.

To show that $\mathcal{T}_{n}$ and $\mathcal{T}_{n+1}$ do not intersect,  we calculate resultant $R_n(x, \te)$ of $\mathcal{T}_n(x)$ and $\mathcal{T}_{n+1}(x)$ in $t$ and check that it does not vanish.

\begin{lemma}
Let  $0 <x<1$, $\lambda >-1/2$, and $\te$ is defined in Theorem \ref{thmain}, then $R_n(x, \te) >0$.
\end{lemma}
\begin{proof}
 The resultant of $\mathcal{T}_n(x)$ and $\mathcal{T}_{n+1}(x)$ in $\tau$ can be written as
$$ R_n(x, \te) =A^2-4 x^2 B C, $$
where
$$A=  n ( n+2\la+1) x^{2 \te}-(n+1) ( n+2\la), $$
$$B= n (n+\la+1) x^\te - (n+1) (n+\la),$$
$$C= ( n+\la) ( n+2\la+1) x^\te - ( n+\la+1) ( n+2\la).$$
We shall consider two cases.

\vspace{0.5cm}
\noindent{\it Case 1.} $\la > 0$.
First, we observe that
$ R_0(x, \te) > 0$ for $\te=\frac{2}{1+2\la}$ and $0 <x<1$.
Indeed,
$$\rho(x)=\frac{1}{4 \la^2} \, R_0(x, \te)=1 - 2(1+\la) x^2  +(1+2\la) x^{2 + \frac{2}{1+2\la}}  >0, $$
since
$$\rho'(x)=4(1+\la)(x^\te-1)x <0,\; \text{and} \;\; \rho(1)=0.$$

Hence, it will be enough to prove the following inequality,
$$
D_n(x, \te)=
\frac{R_n(x, \te)-R_0(x, \te)}{n (n+2\la+1) (1 - x^\theta)}>0,$$
where
$$D_n(x, \te) =
n ( n+2\la+1) (1 - x^\te) (1 - 2 x + x^\te) (1 + 2 x + x^\te) \, + 4 \la \eta(x),$$
and
$$\eta(x)= 1 - (3+\la) x^2  + (1+ x^2 + \la x^2)x^\te .$$
The first summand at the right hand side of $D_n(x, \te)$ is positive as $0<\te <2$ and $0 <x<1$. Let us show that $\eta(x)$ is also positive in that region.
We have
$$\eta(x) >1-(3+\la)x^2+(1+x^2+\la x^2)x^2 =(1-x^2)(1-x^2-\la x^2), $$
hence $\eta(x)>0$ for $x <\frac{1}{\sqrt{1+\la}}$. Assume now that $\frac{1}{\sqrt{1+\la}} \le x <1$, then
$$\frac{(1+2\la )x}{2}\, \eta'(x)=\big(1+2(1+\la)^2 x^2 \big)x^\theta -(2\la^2+7 \la+3)x^2 <$$
$$1+2(1+\la)^2 x^2-(2 \la^2+7\la+3)x^2 =1-(1+3\la)x^2 \le 1-\frac{1+3\la}{1+\la} <0, $$
hence $\eta(x) >\eta(1)=0$.
This completes the proof of Case 1.

\vspace{0.5cm}
\noindent
{\it Case 2.} $-1/2 < \la <0$.
To show that $R_n(x, \te)>0$, we may assume that $B C > 0$, otherwise there is nothing to prove.
By $\te <2$ we have
$$ R_n(x, \te) =A^2-4 x^2 B C \ge A^2-4 x^\te B C =$$
$$z ( n z+2\la) \left( (n+2\la+1) z-2 \la \right)\left( (n+2\la+1)n z-2 \la \right),\;\;\;z=1-x^\te . $$
 This implies that $R_n(x, \te) > 0$ for $\dfrac{(n+1)(n+2 \lambda)}{n(n+2\lambda+1)} < x <1$.

We have to show that $R_n(x,\te) \ge 0$ only for $x_0 < x< 1$, where $x_0$ is defined by (\ref{x0}). Hence, it is enough to check that for $ \te=4/(2-\la)$,
$$x_0^\te >\dfrac{(n+1)(n+2 \lambda)}{n(n+2\lambda+1)}  \, ,$$
or, equivalently,
$$g(n, \la)=2\ln \frac{(n+1)(n+2 \lambda+1)}{(n+\lambda+1)^2}-\lambda \ln \frac{n(n+2\lambda+1)}{(n+1)(n+2\lambda)} >0.$$
Since
$$\frac{\partial g(n, \la) }{\partial n} = - \frac{2 \lambda^2(3n+2 \lambda^2+3 \lambda+1)}{n(n+1)(n+\lambda+1)(n+2\lambda)(n+2 \lambda+1))} <0, $$
we have
$$g(n, \la) > \lim_{n \rightarrow \infty} g(n, \la) =0.$$
This completes the proof.
\end{proof}

\begin{lemma}
\label{taupm}
 Let  $0 <x<1$, $\lambda >-1/2$, then the curves $\mathcal{T}_{n}$ and $\mathcal{T}_{n+1}$ are nested, and $\mathcal{T}_{n+1}$ lies inside $\mathcal{T}_n$.
\end{lemma}
\begin{proof}
For $\lambda >0$, using an obvious notation for the branches of the roots in (\ref{tn}) and (\ref{tn1}), we obtain at $x=1$
$$\left(\tau_n^{(+)},\tau_n^{(-)}\right) =\left( 1, \frac{n}{n+2\la}\, \right), \;\;\;
\left( \tau_{n+1}^{(+)},\tau_{n+1}^{(-)} \right) =\left(1, \frac{n+1}{n+2\la+1}\,  \right).$$
Similarly, for $\la <0$, we obtain
$$\left(\tau_n^{(+)},\tau_n^{(-)}\right) =\left( \frac{n}{n+2\la}\,,1 \right), \;\;\;
\left( \tau_{n+1}^{(+)},\tau_{n+1}^{(-)} \right) =\left( \frac{n+1}{n+2\la+1}\,, 1  \right).$$
Since the curves do not intersect for $x<1$, it follows that $\mathcal{T}_{n+1}$ is enclosed inside $\mathcal{T}_{n}$.
This completes the proof of inequality (\ref{tumain}).

\end{proof}

It is left to show that for $\lambda \ge 0$ the value $\theta = 2/(2 \la+1) $ is sharp for all $n \ge 1$. The following lemma completes the proof of Theorem \ref{thmain}.

\begin{lemma}
Let $\la \ge 0$, then for any $n \ge 1$ inequality (\ref{tumain}) does not hold with $\te > \frac{2}{1+2\la}$ for $x$ less than and sufficiently close to one.
\end{lemma}
\begin{proof}
Expanding $\Delta_n$ into Taylor series around $x=1$ (we used Mathematica), one gets
$${n+2 \la-1 \choose n }^2  \Delta_n=\frac{(2-\te-2\la \te)\Gamma^2(n+2\la)}{(1+2\la) \Gamma^2(2\la) \, n!^2 } \, (1-x) +O((1-x)^2).$$
Hence, $\Delta_n < 0$ in a vicinity of $x=1$ if $\te > \frac{2}{1+2 \la}\, .$
Note that the main term of the expansion for $\te=\frac{2}{1+2 \la}$ is
$$\frac{16 \la^2 (3n^2+6 \la n+2 \la^2-\la) \Gamma^2(n+2\la)}{(3+2\la) \Gamma^2(2 \la+2) n!^2}\, (1-x)^2, $$
which is still positive.
\end{proof}
Since the vertex of $\mathcal{T}_n$ is situated at
$$\tilde{x}=\left( \frac{n(n+2\la)}{(n+\la)^2}\right)^{1/(2-\te)}, \;\;\; \text{where} \;\;\; \frac{1}{2-\te} =\left\{
\begin{array}{cc}
\frac{1}{2}-\frac{1}{\la}\, , & -\frac{1}{2} <\la <0,\\
&\\
\frac{1}{2}+\frac{1}{4\la}\, , & \la \ge 0,
\end{array}
\right.$$
the inequalities of  Theorem \ref{thmain} are non-trivial only in the intervals $(\tilde{x},1)$ and $(-1, -\tilde{x})$. Outside these regions $\Delta_n>0$ just as a positive definite quadratic in $t_n$. In particular, for a fixed $\la$ the length of these intervals shrinks to $O(n^{-2})$.

Using that $\Delta_n$ must be positive definite in the oscillatory region, where $t_n$ runs through all real values, we can provide
more information about the location of vertex $\tilde{x}$ of $\mathcal{T}_n$ with respect to the zeros of $C_{n+1}^{(\la )} (x)$.
\begin{claim}
\label{tilde}
Let $x_1 >x_2 > \ldots > x_{n+1}$,  be the zeros of $C_{n+1}^{(\la )} (x)$. Then
$$
\tilde{x} > \left\{
\begin{array}{cc}
x_2, & -1/2 <\la<0,\\
&\\
x_1, & \la \ge 0.
\end{array}
\right.
$$
\end{claim}
\begin{proof} First notice that $|x_i| <1$.
If $\la=0$ the situation is trivial as $\tilde{x}=1$.
Assuming that $\la \neq 0$, we observe that $\Delta_n >0$ means that $t_n$ does not intersect $\mathcal{T}_n$. Notice also that $t_i$ is an increasing function of $x$ for any $i$, as
  $$y_n^2(x) \frac{d t_i}{d x}=y'_{n+1}(x)y_n(x)-y'_n(x)y_{n+1}(x) >0 $$
  is just the Christoffel-Darboux kernel. It is also easy to check that $\mathcal{T}_n^{(-)}(x)$, the lower branch of $\mathcal{T}_n$,  is positive and decreasing in $x$.
  The situation when $x_2  \ge \tilde{x}$ is impossible, as then $t_n$ running through the whole interval $[0,\infty)$ intersects $\mathcal{T}_n^{(-)}(x)$ at some point between $x_2$ and $x_1$. The same is true regarding the interval $[x_1,1]$, provided the intersection occurs before $x=1$. For $x \in (x_1,1]$, $t_n(x)$ increases from zero to one. However $\tau_n^{(-)}(1)=\frac{n+\la-|\la|}{n+2 \la} <1$ for $\la >0$, where  $\tau_n^{(\pm)}$ are defined by (\ref{tn}). Therefore the case $x_1 \ge \tilde{x}$ is impossible for $\la >0$, as otherwise $t$ and $\mathcal{T}_n^{(-)}(x)$ would intersect inside the interval $(x_1,1)$.
\end{proof}
The different situations appearing in the proof of Lemma \ref{tilde} are illustrated in Fig.1 for $n=4$ and three  values of $\la$. Here $x_2 <\tilde{x}< x_1$ for $\la=- \frac{1}{3}$, and $x_1 < \tilde{x}$ for $\la=-\frac{1}{4}$ and $\la=1$. In general, the difference between cases $\la <0$ and $\la >0$ in Lemma \ref{tilde} occurs since $t$ passes below $\mathcal{T}_n$ for $\la <0$ and above it for $\la >0$. For $\la <0$ we have $\tau_n^{(-)}(1)=1$, and both situations, $x_2 <\tilde{x} < x_1$ and $x_1 < \tilde{x} \le 1$, may occur depending on the values of $\la$ and $n$.

\begin{remark}
It seems difficult to find the exact value of $\te$ for $ -\frac{1}{2} < \la <0$. Let us note that for $\te >\frac{8}{4-\la}$ the resultant changes the sign for sufficiently large $n$. Thus the curves intersect and our arguments fail. Indeed, for $\te >\frac{8}{4-\la}$, in the notation of Lemma \ref{taupm},  the Taylor series expansion around $ n =\infty$ at $\hat{x}=\frac{3x_0+1}{4}$ yields,
$$\tau_n^{(+)}(\hat{x})-\tau_{n+1}^{(+)}(\hat{x})=\frac{3\lambda \left((4-\lambda)\te-8 \right))}{4(2-\te)n^2} +O(n^{-3}) <0,$$
$$\tau_{n+1}^{(-)}(\hat{x})-\tau_{n}^{(-)}(\hat{x})=\frac{\lambda \left((4+3 \lambda)\te-8 \right)}{4(2-\te)n^2}+O(n^{-3}) >0.$$
For $\te =\frac{8}{4-\la}$ the curves probably do not intersect as
$$R_n(\hat{x} ,\,\frac{8}{4-\la})=\frac{9}{2} \,(8+\la^2) \la^4 n^{-4} +O(n^{-5})>0.$$
\end{remark}

\section{Proof of Theorem \ref{genleg}}

\begin{proof}
It is enough to consider the case $0<x<1$, as in Theorem \ref{thmain}. We may assume that coefficients $a_n$ are strictly decreasing treating the general situation as the limiting case.
By the assumption $\te <2$, we get
$$\Delta_1=x^{2\te}-\frac{x^2-a_1}{1-a_1} >\frac{1-x^2}{1-a_1} (2a_1-1) >0.$$

 Note also that if the \tu inequality of the theorem holds for some $\te_0$, it also holds for any $\te <\te_0$.

 The proof is by the induction on $n$ and is similar to that of Theorem \ref{thmain}. To show that $\Delta_{n+1} \ge 0$ we consider the following two expressions,
 $$\dfrac{a_n \Delta_n}{p_n^2(x)}\;\;\;\mbox{and}\;\;\; \dfrac{(1-a_{n+1}) \Delta_{n+1}}{p_n^2(x)}\,,$$ set $t=p_{n+1}(x)/p_n(x) ,$ and define the corresponding curves
$$\mathcal{T}_n (x)=(1-a_n)\tau^2 - x \tau  + a_n  \tau^\te =0, $$
$$\mathcal{T}_{n+1} (x) =(1-a_{n+1)}x^\te \tau^2-x \tau+a_{n+1} =0,$$
with the vertices
$$\left(4a_n-4a_n^2 \right)^{1/(2-\te)} \; \text{and} \; \left(4a_{n+1}-4a_{n+1}^2 \right)^{1/(2-\te)},$$
respectively.
The both curves have a single component and are simultaneously real for
$$x \ge x_0=\left(4a_{n+1}-4 a_{n+1}^2 \right)^{1/(2-\te)} .$$

First, we have to show that $\mathcal{T}_n$ and $\mathcal{T}_{n+1}$ do not intersect on $(0,1)$.
The resultant of $\mathcal{T}_n$ and $\mathcal{T}_{n+1}$ in $\tau$ is given by
\be
\label{res}
R_n(x,\te)=A^2-x^2 B C,
\ee
where
$$A=a_n (1-a_{n+1})t^{2 \te}-(1-a_{n})a_{n+1} , $$
$$B=a_{n+1}-a_n x^\te , $$
$$C=1-a_{n}-(1-a_{n+1})x^\te. $$
If $BC \le 0$ there is nothing to prove. Otherwise, using $x^2<x^\te$, we get
$$R_n(x,\te) > A^2-x^\te BC =$$
$$
(1-x^\te)\left(1-a_n -a_n x^\te \right)\left( a_{n+1}-(1-a_{n+1})x^\te \right)\left(a_{n+1} (1-a_{n})-a_n (1-a_{n+1})x^\te \right).$$
Thus, $R_n(x,\te) >0$ for
$$\frac{a_{n+1}(1-a_n)}{a_n(1-a_{n+1})}< x^\te <1. $$
Hence, the curves do not intersect on $(x_0,1)$ if
$$x_0^\te > \frac{a_{n+1}(1-a_n)}{a_n(1-a_{n+1})} \, .  $$
This yields inequality (\ref{gen te}).

It is left to check that $\mathcal{T}_{n+1} (x)$ lies inside $\mathcal{T}_{n} (x)$.
 For $x=1$, one finds
 $$\mathcal{T}^{(+)}_n(1)=\frac{a_n}{1-a_n} > \mathcal{T}_{n+1}^{(+)}(1) = \frac{a_{n+1}}{1-a_{n+1}} \, , \;\;\; \mathcal{T}_n^{(-)}(1) = \mathcal{T}_{n+1}^{(-)}(1) =1.$$
Hence, the curves are nested with $\mathcal{T}_{n+1}$ lying inside $\mathcal{T}_{n}$.
\end{proof}

\begin{remark}
It is easy to show that for ultraspherical polynomials and $\la <0$, Theorem \ref{genleg} leads to the same result as Theorem \ref{thmain}.
Indeed, for the ultraspherical case, $a_n=\frac{n}{2(n+\la)} $,
and
$$F(n)=\dfrac{2 \log \dfrac{(1-a_n)\,a_{n+1}}{(1-a_{n+1})\,a_n}}{\log \dfrac{4(1-a_n)\, a_{n+1}^2}{a_n}}=
\dfrac{2\log \dfrac{(n+1)(n+2\la)}{n(n+2\la+1)}}{\log \dfrac{(n+1)^2(n+2\la)}{n(n+\la+1)^2}} \, .$$
One finds $\frac{d F(n)}{d n} <0$, hence,
$$ \inf F(n)=\lim_{n \rightarrow \infty}F(n) =\frac{4}{2-\la} \, .$$
On the other hand, an attempt to generalise the case $\la >0$, would require to work with recurrence (\ref{gen3term}) for a growing sequence $a_n$, where $a_n <1/2$. It is unclear how to extract $\te$ from the inequality $R_n(x,\te) >0$ in that case.
\end{remark}

\section{Proof of Theorem \ref{thherm}}

\begin{proof}
The proof is by the induction on $n$.
 Without loss of generality, we assume that the sequence $a_i$ is strictly increasing and that $x >0$.
It will be convenient to set $a_0=0$, then $\Delta_1(x)=a_1^2/(x^2+a_1)>0$.
Rewriting $\Delta_n$ and $\Delta_{n+1}$ in terms of $t=\frac{p_{n+1}(x)}{p_n(x)}$, and setting $ d_i=a_i-a_{i-1}>0$, we
consider $a_n(x^2+d_n)\dfrac{\Delta_n(x)}{p_n^2(x)}$ and $(x^2+d_{n+1})\dfrac{\Delta_{n+1}(x)}{p_n^2(x)}$.
Then the corresponding curves are
$$\mathcal{T}_n(x)=(x^2+d_n)\tau^2-(x^2+d_n)x \tau +a_n x^2 =0,$$
$$\mathcal{T}_{n+1}(x)=x^2 \tau^2-(x^2+d_{n+1})x \tau+a_{n+1}(x^2+d_{n+1})=0.$$
To show that $\mathcal{T}_n$ and $\mathcal{T}_{n+1}$ do not intersect, we
calculate the resultant in $\tau$. This yields a rather long expression of the form
$$ \sum_{i=0}^3 b_{i}\, x^{2i},$$
where all $b_{i}$ are positive as they are composed of positive monomials. Therefore the resultant does not vanish. We omit the details.

It is left to verify that the curves are nested and $\tau_{n+1}$ lies inside $\tau_n$.
For, it is enough to check that the value of $\mathcal{T}_n$ at
the vertex of $\mathcal{T}_{n+1}$ with coordinates $$\left( \sqrt{3 a_{n+1}+a_n}\, , \dfrac{2a_{n+1}}{\sqrt{3 a_{n+1}+a_n}} \right)$$ is negative.
Indeed, one calculates that it is
$$- \frac{6d_{n+1}^3+(17a_n+2d_n)d_{n+1}^2+6 a_n (2 a_n + d_n) d_{n+1}+4 a_n^2 d_n}{3a_{n+1}+a_n} <0. $$
Hence, $\Delta_{n+1}(x)>0$, provided $\Delta_n(x) >0$. This completes the proof of (\ref{genhermite}).

Finally, let $H_n(x)$ and $\mathcal{H}_n(x)$ be the Hermite polynomials in the standard and monic normalisation, respectively, $\mathcal{H}_n(x)=2^{-n} H_n(x).$
Now (\ref{hermite}) follows by $a_n=n/2$ in the monic case.

\end{proof}

\end{document}